\documentclass[11pt]{amsart}
\usepackage{amsmath,amssymb,verbatim,eucal, amscd,color}
\usepackage{graphicx}
\usepackage{colordvi}
\usepackage{xypic}
\usepackage[colorlinks=true, pdfstartview=FitV,
linkcolor=blue, citecolor=blue, urlcolor=red]{hyperref}
\usepackage{amsfonts,latexsym}
\usepackage{mathrsfs}

\usepackage{geometry,tikz,amsmath}

\usetikzlibrary{fit,matrix}

\input xy

\numberwithin{equation}{section}

\theoremstyle{plain}
				
\newtheorem{theorem}{Theorem}[section]				
\newtheorem{proposition}[theorem]{Proposition}		
\newtheorem{corollary}[theorem]{Corollary}
\newtheorem{lemma}[theorem]{Lemma}

\theoremstyle{definition}

\newtheorem{definition}[theorem]{Definition}
\newtheorem{remark}[theorem]{Remark}

\newcommand{\CBbb}{\mathbb C}

\newcommand{\VBbb}{\mathbb V}

\newcommand{\ZBbb}{\mathbb Z}

\newcommand{\Bcal}{\mathcal B}

\newcommand{\Ical}{\mathcal I}

\newcommand{\Lcal}{\mathcal L}
\newcommand{\Mcal}{\mathcal M}

\newcommand{\Ocal}{\mathcal O}

\newcommand{\SO}{\mathsf{SO}}
\newcommand{\Sp}{\mathsf{Sp}}
\newcommand{\PSp}{\mathsf{PSp}}
\newcommand{\Spin}{\mathsf{Spin}}
\newcommand{\PSO}{\mathsf{PSO}}

\newcommand{\slfrak}{\mathfrak{sl}}

\newcommand{\SC}{\mathsf{SC}}

\DeclareMathOperator{\imag}{Im}

\DeclareMathOperator{\tr}{tr}

\DeclareMathOperator{\Div}{div}

\DeclareMathOperator{\Nm}{Q}
\DeclareMathOperator{\nm}{Nm}

\newcommand{\lra}{\longrightarrow}

\newcommand{\Pic}{{\rm Pic}}

\makeatletter
\newcommand*\bigcdot{\mathpalette\bigcdot@{.5}}
\newcommand*\bigcdot@[2]{\mathbin{\vcenter{\hbox{\scalebox{#2}{$\m@th#1\bullet$}}}}}
\makeatother

\makeatletter
\newcommand{\thickcolon}{\mathpalette\thick@colon\relax}
\newcommand{\thick@colon}[2]{%
  \mspace{1mu}%
  \vbox{%
    \hbox{$\m@th#1\bigcdot$}
    \nointerlineskip
    \kern.15ex
    \hbox{$\m@th#1\bigcdot$}
    \kern-.55ex
  }%
  \mspace{1mu}%
}
\makeatother

\AtBeginDocument{%
  \def\MR#1{}
  }

\begin{document}
\title[Spectral data for Spin Higgs Bundles]{Spectral data for Spin Higgs Bundles}
\author{Swarnava Mukhopadhyay}

\address{Tata Institute of Fundamental Research,
Homi Bhabha Road, 
Mumbai 400005, INDIA}
%
\email{swarnava@math.tifr.res.in}
\author{Richard Wentworth}
\address{Department of Mathematics, University of Maryland, College Park, MD 20742, USA}
\email{raw@umd.edu}
 \thanks{S.M. thanks the Max Planck Institute for Mathematics in Bonn, where a portion of this work was completed. R.W.'s research was supported by grants from the National Science Foundation. The authors also acknowledge support from NSF grants DMS-1107452, -1107263, -1107367 ``RNMS: GEometric structures And Representation varieties'' (the GEAR Network). 
  } 
\subjclass[2010]{Primary  14H60, Secondary 17B67,  32G34, 81T40}
\begin{abstract} 
In this paper we determine the spectral data parametrizing Higgs bundles in a generic fiber of the Hitchin map for the case where the structure group is the special Clifford group with fixed Clifford norm. These are spin and ``twisted'' spin Higgs bundles. The method used relates variations in spectral data with respect to the Hecke transformations for orthogonal bundles introduced by Abe.
The explicit description also recovers a result from the geometric Langlands program which states that  the fibers of the Hitchin map are the dual abelian varieties to the corresponding fibers of the moduli spaces of projective orthogonal Higgs bundles (in the even  case) and  projective symplectic Higgs bundles (in the odd  case).
\end{abstract}

\maketitle

\allowdisplaybreaks

\section{Results}
Let $X$ be a smooth projective algebraic curve of genus $g\geq 2$ and $p\in X$.
Let
$\Mcal^{\pm}_{\Spin(N)}(X)$ denote the coarse moduli spaces of semistable Higgs bundles on $X$ with the special Clifford group $\SC(N)$ as  structure group  and fixed Clifford norm of even $(+)$ or odd $(-)$ degree, respectively. 
For concreteness and without loss of generality, we require the Clifford norms to be $\Ocal_X$ and $\Ocal_X(p)$, respectively.
Then $\Mcal^{+}_{\Spin(N)}$ is exactly the moduli space of $\Spin(N)$ Higgs bundles, whereas $\Mcal^-_{\Spin(N)}$ is a moduli space of \emph{twisted} $\Spin(N)$ Higgs bundles (see Section \ref{sec:clifford}). 

The Hitchin fibration takes the form:
$$
h^\pm_{\Spin(N)}: \Mcal^{\pm}_{\Spin(N)}\lra \Bcal(N):=
\begin{cases}
\bigoplus_{i=1}^{m} H^0(X, K_X^{2i}) & N=2m+1 \\
\bigoplus_{i=1}^{m-1} H^0(X, K_X^{2i})\oplus H^0(X, K_X^m) & N=2m\ ,
\end{cases}
$$
The maps $h^\pm_{\Spin(N)}$ realize these moduli spaces as  algebraically complete integrable systems
whose generic fibers are torsors over abelian varieties.
 The main goal of this note is to describe these abelian varieties explicitly in terms of \emph{spectral data}.

 The problem is clearly related to the case of Higgs bundles for  orthogonal groups. Here there is a complete description (see \cite{Hitchin:87b, Hitchin:07}). Hitchin describes the spectral data  in terms of line bundles in the \emph{Prym variety} associated to the spectral curve  defined by the point in $\Bcal(N)$. The construction, which we briefly review in Sections \ref{sec:spectraloddorthogonal} and \ref{sec:spectralevenorthogonal} below, involves fixing a spin structure on $X$. In the end, this ancillary choice is irrelevant, but it gives a hint that hidden in the argument is actually a lift to $\Spin(N)$ (or $\SC(N)$). 
 We shall show that these data  indeed provide the extra structure of a Clifford  bundle. 
 
  To be more precise, let  $\vec b\in \Bcal(N)$ be a generic point.
  By the \emph{spectral curve} $\pi : S\to X$  we mean (somewhat unconventionally)
  the normalization of the branched cover of $X$ defined by $\vec b$ (see
  Sections \ref{sec:spectraloddorthogonal} and \ref{sec:spectralevenorthogonal}).
  Let $\overline S=S/\sigma$, where $\sigma$ is the natural involution, and let $K(S,\overline S)$ denote the kernel of the norm map $\nm_{S/\overline S}:J(S)\rightarrow J(\overline{S})$.
  Then $K(S,\overline S)$ is just the Prym variety $P(S,\overline S)$ of the cover $p: S\to \overline S$ for $N$ odd, whereas for $N$ even, $P(S,\overline S)\subset K(S,\overline S)$ is the connected component of the trivial bundle. 
 In both cases,  $J_2(\overline S)$ acts additively on $K(S,\overline S)$ by pulling back via $p^\ast$,  and it acts  on $J_2(X)$ via the norm map of the covering $\overline S\to X$.
 The main result may then be stated as follows.

\begin{theorem}[{\sc Spectral Data}] \label{thm:main}
For generic points $\vec b\in \Bcal(N)$, the  fiber  $(h^\pm_{\Spin(N)})^{-1}(\vec b)$  is a torsor over the abelian variety:
\begin{equation} \label{eqn:A-spin}
A_{\Spin(N)}(X,\vec b):=
K(S,\overline S)\times_{J_2(\overline{S})} J_2(X) \  .
\end{equation}
In terms of Prym varieties,
$$
A_{\Spin(N)}(X,\vec b)=
\begin{cases}
P(S,\overline S)\times_{J_2(\overline{S})} J_2(X)\ , & N \text{ odd };\\
P(S,\overline S)\times_{H_1/H_0} J_2(X)\ , & N \text{ even },
\end{cases}
$$
where in the even case $H_0\simeq \ZBbb/2$ is the subgroup of $J_2(\overline S)$ generated by the line bundle defining the \'etale cover $S\to \overline S$, and $H_1$ is the annihilator of $H_0$ in $J_2(\overline S)$  with respect to the Weil pairing.
\end{theorem}

The fact that the right hand side of \eqref{eqn:A-spin} is connected is not quite obvious (see Lemma \ref{lem:connected}). 
We have the following consequence.
\begin{corollary} \label{cor:connected}
The fibers of the Hitchin map for $\Mcal^-_{\Spin(N)}$ are connected.
\end{corollary}

In the case of $\Mcal^{+}_{\Spin(N)}$, i.e.\ $\Spin(N)$-Higgs bundles, 
the connectedness of the fibers follows from a general result of several authors (cf.\ \cite{DonagiPantev:12,Faltings:93}), whereas the fact that the fiber structure is the same in the twisted case $\Mcal^{-}_{\Spin(N)}$ (and hence also connected)  is a consequence of Theorem \ref{thm:main}.
An application of this fact is the following:
 Hitchin's construction  \cite{Hitchin:90} of a projectively flat connection on the space of generalized $\Spin(N)$ theta functions works as well in the twisted case. 
This connection  was used by the authors in \cite{MukhopadhyayWentworth:16} in their study of  strange duality for odd orthogonal bundles. 

 For a general complex reductive Lie group $G$, work of Donagi-Pantev \cite{DonagiPantev:12}, Hitchin \cite{Hitchin:07}, and Hausel-Thaddeus \cite{HauselThaddeus:03} show that the Hitchin system associated to $G$ is dual to the Hitchin system associated to ${}^LG$, the  Langlands dual group to $G$. Another consequence of Theorem \ref{thm:main} is an explicit duality for spin bundles.
Let $A_{\PSO(2m)}(X,\vec b)$ and $A_{\PSp(2m)}(X,\vec b) $ denote the fibers over $\vec b\in \Bcal(N)$ of the Hitchin map for the moduli spaces of projective special orthogonal and projective symplectic Higgs bundles, where $N=2m$ or $N=2m+1$, respectively. Then we have the following theorem (see Section \ref{sec:duality}):
\begin{theorem}[{\sc Langlands Duality}] \label{thm:duality}
For generic points $\vec b$, 
we have the following dualities of abelian varieties:
\begin{enumerate}
\item $\left(A_{\Spin(2m)}(X,\vec b)\right)^\vee \simeq A_{\PSO(2m)}(X,\vec b) $;
\item $\left(A_{\Spin(2m+1)}(X,\vec b)\right)^\vee \simeq A_{\PSp(2m)}(X,\vec b) $.
\end{enumerate}
\end{theorem} 

Here is a brief sketch of the main idea behind the proof of Theorem \ref{thm:main}. First,  spectral data describe an orthogonal bundle $V_L\to X$ in terms of a line bundle $L\in K(S,\overline S)$.  
In Section \ref{sec:hecke} we show that if $L$ is modified by a line bundle defined by a generic point  $p\in S$ (and its reflection by $\sigma$), then the new orthogonal bundle obtained is exactly the \emph{Hecke transformation} of $V_L$ at the point $\pi(p)$ introduced by Abe \cite{Abe:13}. The result, Corollary \ref{cor:hecke}, means that we can move around in the spectral data for orthogonal bundles via elementary  transformations on the bundle itself. This interpretation makes  it  transparent that a choice of lift of $V_L $ to a Clifford bundle induces a lift  on the transformed bundle as well (see Corollary \ref{cor:spin-hecke}).
In this way, a Clifford structure is naturally defined on $V_L$, $L=M^2$, given one on the orthogonal bundle with ``trivial'' spectral data. 
 We then show that the dependence of this structure on   $M$ is exactly given by the action of $J_2(\overline S)$ via the norm map.
 In Section \ref{sec:duality}, we prove directly that the abelian varieties appearing are dual to the ones for the projective symplectic and orthogonal cases.

 \medskip\noindent 
 {\bf Acknowledgments.} The authors warmly thank Steve Bradlow, Lucas Branco,  and Nigel Hitchin for discussions related to this work. 

\section{Preliminaries}

\subsection{Clifford bundles} \label{sec:clifford}

Let $(\VBbb,q)$ be a complex orthogonal vector space, $C(\VBbb)$  the Clifford algebra of $(\VBbb,q)$, and $C_+(\VBbb)$ the even part.
The \emph{special Clifford group} is defined as follows:
$$\SC(\VBbb)=\left\{ g\in  C_+^\times(\VBbb) \mid gvg^{-1}\in \VBbb\text{ for all } v\in \VBbb\right\}\ .$$
The induced action of $\SC(\VBbb)$ on $\VBbb$ is by orthogonal transformations and gives rise to  an exact sequence
\begin{equation} \label{eqn:clifford-sequence}
0\lra \CBbb^\times\lra \SC(\VBbb) \lra \SO(\VBbb)\lra 0\ .
\end{equation}
The \emph{Clifford (or spinor) norm} of an element $g\in \SC(\VBbb)$ is defined as
\begin{equation*}\label{eqn:clifford-norm}
\Nm(g)=q(v_1)\cdots q(v_k)\ ,
\end{equation*}
where $g=v_1\cdots v_k$, for $v_j\in \VBbb$ (any $g\in\SC(\VBbb)$ has such an expression). The \emph{spin group} is then $\Spin(\VBbb):=\Nm^{-1}(1)$. The restriction of \eqref{eqn:clifford-sequence} to $\Spin(\VBbb)$ becomes
\begin{equation*} \label{eqn:spin-sequence}
0\lra \ZBbb/2\lra \Spin(\VBbb) \lra \SO(\VBbb)\lra 0\ .
\end{equation*}
We set $\SC(N)=\SC(\CBbb^N)$, where $\CBbb^N$ has  the standard orthogonal structure. 

For a connected complex reductive Lie group $G$,  let
 $\Mcal_{G}$ denote the coarse moduli space of semistable $G$-Higgs bundles on $X$. In the case $G=\SC(N)$, the Clifford norm 
 induces a morphism  $\mathcal{M}_{\SC(N)}\rightarrow \Pic(X)$, which  we also denote  by $\Nm$.
For an $\SC(N)$ bundle $P$ and $L\in \Pic(X)$, we will denote by $P\otimes L$ the $\SC(N)$ bundle whose transition functions are  obtained by multiplying those of $P$ and $L$. It is then clear that:
\begin{equation} \label{eqn:square}
\Nm(P\otimes L)= \Nm(P)\otimes L^2\ .
\end{equation}

 Fix $p\in X$, and consider bundles $\mathcal{O}_X(dp)$, where $d\in \ZBbb$. 
 Then the preimage by $\Nm$ of the class of
 $[\Ocal_X(dp)]\in \Pic(X) $ depends only on the parity of $d$.
 Let
 $\Mcal^\pm_{\Spin(N)}$ be  the inverse images of the bundles $[\mathcal{O}_C(dp)]$, for $d=0,1$, respectively. 
Therefore, while by definition $\Mcal^+_{\Spin(N)}=\Mcal_{\Spin(N)}$, the space $\Mcal^-_{\Spin(N)}$ is a ``twisted'' component that does not correspond to a moduli space of $G$-bundles for any complex reductive $G$.
The connected components  $\Mcal^\pm_{\SO(N)}$ of $\Mcal_{\SO(N)}$ are labeled by the second Stiefel-Whitney class: $V\in  \Mcal^\pm_{\SO(N)} \iff w_2(V)=\pm 1$ (cf.\ \cite[Prop.\ 1.3]{BLS:98}),
and the  projection  \eqref{eqn:clifford-sequence}  induces a morphism 
$\Mcal^\pm_{\Spin(N)} \to \Mcal^\pm_{\SO(N)}\ .$
We refer to \cite[Prop.\ 3.4]{Oxbury:99} for more details.

In this paper, we mostly regard points in $\Mcal^\pm_{\SO(N)}$ as equivalence classes of rank $N$ semistable orthogonal Higgs bundles: i.e.\ a holomorphic 
bundle $V\to X$ with nondegenerate symmetric bilinear pairing  $(\, ,\, ):V\otimes V\to \Ocal_X$, and a fixed isomorphism $\det V\simeq \Ocal_X$, equipped with a holomorphic map $\Phi: V\to V\otimes K_X$ satisfying $(\Phi v, w)+(v,\Phi w)=0$ for all $v, w\in V$.

\subsection{Hecke transformations of orthogonal bundles} \label{sec:hecke-orthogonal}We first recall Hecke transformations for orthogonal bundles following T. Abe \cite{Abe:08}. Let $V\to X$ be an orthogonal bundle.
Choose a  point $p\in X$ and an isotropic line $\tau$ in the fiber $V_p$ of $V$ at $p$. 
Let $\tau^\perp$
denote the orthogonal subspace to $\tau$ in $V_p$, and set 
$\tau_1=V_p/\tau^\perp$. We view $\tau$ and $\tau_1$ as torsion sheaves on $X$ supported at $p$.
 Then we may define a locally free sheaf $ V^\flat\to X$ by the elementary transformation:
\begin{equation} \label{eqn:E-flat}
0\lra  V^\flat\to  V\to\tau_1\lra 0\ .
\end{equation}
Next, let $ V^\sharp=( V^\flat)^*$.
Since the orthogonal structure gives an isomorphism $ V^*\simeq V$, dualizing \eqref{eqn:E-flat} yields an exact sequence:
\begin{equation} \label{eqn:E-flat-dual}
0\lra  V\to  V^\sharp\to \tau_2\lra 0\ ,
\end{equation}
where $\tau_2$ is a torsion sheaf supported at $p$ of length $1$.
 Now the orthogonal structure also induces  maps
\begin{equation}\label{eqn:q-sharp}
V^\sharp\otimes  V^\sharp\lra \Ocal_X(p)\ \ \mbox{and} \
V^\flat\otimes  V^\sharp\lra \Ocal_X\ .
\end{equation}
Consider the subsheaf $ V^\flat\hookrightarrow  V^\sharp$ obtained by 
composing $V^\flat\hookrightarrow V$ in \eqref{eqn:E-flat}
with $V\hookrightarrow V^\sharp$ in
 \eqref{eqn:E-flat-dual}.
Then $ V^\sharp/ V^\flat$ is a torsion sheaf supported at $p$, and the fiber at $p$ is a rank $2$   orthogonal space.   Since $ V/ V^\flat$ is isotropic, there is a canonical splitting $ V^\sharp/ V^\flat\simeq  V/ V^\flat\oplus \tau_2$.  Finally, we define  $ V^\iota\subset  V^\sharp$ to be the kernel of the map $ V^\sharp\to  V/ V^\flat$. Equivalently, there is an exact sequence
\begin{equation} \label{eqn:E-iota}
0\lra  V^\flat\to  V^\iota\to \tau_2\lra 0\ .
\end{equation}
Then $ V^\iota$ inherits an orthogonal structure from \eqref{eqn:q-sharp}. Moreover, the exact sequence \eqref{eqn:E-iota} determines an isotropic line $\tau^\iota$ in the fiber of $V^\iota$ at $p$. Finally, from \eqref{eqn:E-flat} and \eqref{eqn:E-iota},  the trivialization of $\det V$ induces one for $\det V^\iota$.
\begin{definition} \label{def:iota}
The $\iota$-transform is the map: $( V, \tau )\mapsto ( V^\iota,  \tau^\iota)$.
\end{definition}

\subsection{Hecke transformations of Clifford bundles} \label{sec:hecke-spin}

We wish to extend the previous discussion to Clifford bundles.
For this it is useful to have a description of the $\iota$-transform  explicitly in terms of transition functions. 
  Let
 $P$ be a special Clifford  bundle and $V$ the associated
orthogonal bundle.
Let $\Delta\subset X$ be a disk centered at $p$, and $\sigma: \Delta\to P$ a section. Set $X^\ast=X\setminus\{p\}$.
This gives a trivialization of $P$ and a local frame $e_1,\ldots, e_N$ for $V$ on $\Delta$ with respect
to which the quadratic structure is, say, of the form $(e_i, e_j)=\delta_{i+j-1,N}$.
Similarly, we may choose a section of $\left.{P}\right|_{X^\ast}$. Set $\Delta^\ast=X^\ast\cap \Delta$.
Let $\widehat\varphi: \Delta^\ast\to \SC(N)$ denote the transition function gluing the  bundles $\left.{P}\right|_\Delta$ and $\left.{P}\right|_{X^\ast}$, 
and let $\varphi: \Delta^\ast\to \SO(N)$ be the quotient transition function for 
 $({\mathcal E},q)$.
The transformed bundle ${V}^\iota$  is defined by modifying  $\varphi$  
 by
$\zeta: \Delta^\ast\to \SO(N)$, where
\begin{equation} \label{eqn:zeta}
\zeta=\left(\begin{matrix}
z&&&&\\
&1&&&\\
&&\ddots &&\\
&&&1&\\
  &&&&z^{-1}
\end{matrix}\right)\ .
\end{equation}
Write $z=\exp(2\pi i\xi)$, $\imag \xi>0$,  and set 
\begin{equation} \label{eqn:zeta-hat}
\widehat \zeta(z)= \exp(\pi i\xi)
\exp\left((\pi i\xi/2)(e_1e_N-e_Ne_1)
\right)\ .
\end{equation}
One checks that $\widehat \zeta$ is well-defined under $\xi\mapsto\xi+1$,
and so it yields a map $\widehat \zeta: \Delta^\ast\to \SC(N)$.
Moreover,
the projection \eqref{eqn:clifford-sequence} of $\widehat \zeta$ to $\SO(N)$
 recovers $\zeta$.
Gluing  the trivial $\SC(N)$-bundles over $\Delta$ and $X^\ast$
via $\widehat\varphi(z) \widehat \zeta(z)$, 
we define a new Clifford bundle ${P}^\iota$.
The 
associated orthogonal bundle (with transition function $\varphi(z)\zeta(z)$) coincides with ${V}^\iota$.
With this understood, the main observation is the following.
\begin{proposition}
We have: $\Nm(P^\iota)\simeq \Nm(P)\otimes \Ocal_X(p)$. In particular,
$w_2( V^\iota)=-w_2( V)$.
\end{proposition}

\begin{proof}
From 
\eqref{eqn:zeta-hat},  $\Nm(P^\iota)$ is a line bundle with transition function on $\Delta^\ast$ given by:
$$
\Nm(\widehat\varphi\widehat \zeta)=\exp(2\pi i\xi)\Nm(\widehat\varphi(z)) \Nm
\left(\exp\left((\pi i\xi/2)(e_1e_N-e_Ne_1)\right)\right) = z\cdot\Nm(\widehat\varphi(z))  \ ,
$$
since $\exp\left((\pi i\xi/2)(e_1e_N-e_Ne_1)\right)\in \Spin(N)$; 
hence, the result.
\end{proof}

\begin{remark} \label{rmk:sign}
We could just as well have chosen a prefactor of $\exp(-\pi i\xi)$ in \eqref{eqn:zeta-hat} to obtain a Clifford bundle with norm $\Nm(P)\otimes \Ical_p$. The two Clifford bundles thus defined are isomorphic after tensoring by $\Ocal_X(p)$.
Later on, however,  we shall consider multiple  Hecke transformations at points of a reduced, not necessarily effective divisor $D$; hence, the points will have a sign. For convenience, at
 each point $p$ in the support of $D$ we shall choose $\exp(\pm\pi i\xi)$ in \eqref{eqn:zeta-hat} so that the change of Clifford norm corresponds to the sign of $p$. 
\end{remark}

\subsection{Spectral data for symplectic bundles} \label{sec:spectral-symplectic}
In this section, following \cite{Hitchin:07} and \cite{Hitchin:87b}, we briefly recall the explicit description of generic  fibers of the Hitchin map for the groups $\Sp(2m)$ and $\PSp(2m)$ in terms of spectral data . 
Let $E\to X$ be a symplectic bundle of rank $2m$ with pairing $\langle\, ,\,\rangle$, and $\Phi: E\rightarrow E\otimes K_X$ be a Higgs field such that $\langle v,\Phi w\rangle+\langle\Phi v,w\rangle=0$,
 for all sections $v, w\in E$. The coefficients of the characteristic polynomial of $\Phi$ gives an element $\vec b\in \mathcal{B}(N)$, 
 $N=2m+1$. 
Let $\pi: {\rm tot} (K_X)\to X$ be the total space of the canonical bundle on $X$, and
let $\lambda: {\rm tot} (K_X)\to \pi^\ast K_X$ denote the tautological section. We assume the last coefficient $b_{2m}$ of $\vec b$ has simple zeros at $Z(b_{2m})\subset X$, and that $b_{2m-2}$ is nonzero on $Z(b_{2m})$.  The spectral curve associated to $\vec b$ is:
\begin{equation} \label{eqn:spectral-curve}
S :=\left\{ w\in {\rm tot} (K_X) \mid \lambda^{2m}+\pi^\ast(b_2)\lambda^{2m-2}+\cdots +\pi^\ast(b_{2m})=0 \text{ at } w \right\}
\end{equation}
The assumptions guarantee  that $S$ is smooth (see \cite{BNR}). Let $\sigma$ denote the involution $w\mapsto -w$ on $S$, and $\overline S=S/\sigma$.
In this case, the kernel $K(S,\overline S)$ of the norm map $\nm_{S/\overline S} : J(S)\to J(\overline S)$ is connected and so coincides with the Prym variety $P(S,\overline S)$.
 Let $L\in P(S,\overline S)$, and consider $U=L\otimes \left(K_S\otimes \pi^\ast K_X^{-1}\right)^{1/2}$. Since $K_S=\pi^\ast K_X^{2m}$, a square root of $K_S\otimes \pi^*K_X^{-1}$ can be given by a choice of theta characteristic on $X$, which we fix once and for all. We then have the following result of Hitchin.
\begin{theorem}[{\cite[Sec.\ 5.10]{Hitchin:87b}}]             \label{theorem:spspectral}
	Let $\vec b \in \mathcal{B}(N)$ be such that $b_{2m}$ has simple zeros and $b_{2m-2}$ is nonzero on $Z(b_{2m})$.	
	Then $h^{-1}_{\Sp(2m)}(\vec{b})$ identifies with points on the Prym variety $P(S,\overline{S})$. The correspondence sends $L\in P(S,\overline S)$ to 
$E=\pi_*(U)$, where 
	$$U=L\otimes (K_S\otimes \pi^*K_X^{-1})^{-1/2}\in P(S,\overline{S})\ .$$
	The Higgs field $\Phi$ is obtained by multiplication with $\lambda$. Conversely,  given  a symplectic higgs bundle $({E},\Phi)$,
	let $U\to S$ be the line bundle 
	$$U:=\ker(\pi^*E \stackrel{\Phi}{\lra} \pi^*(E\otimes K_X))\ .$$
	Then $L=U\otimes (K_S\otimes \pi^*K_X^{-1})^{1/2}\in P(S,\overline S)$.
\end{theorem}
 The symplectic structure on $E=\pi_\ast U$ is defined as follows.
 Let $\widehat\sigma$ denote a linearization of $\sigma$ on $U$. 
 Then $(\widehat\sigma)^2$ acts as $-1$ on the fibers.   It suffices to define a nondegenerate skew-pairing on sections of $U$ over open sets $A\subset X$.  For two such sections $u,v$, let
\begin{equation} \label{eqn:symp-pairing}
\langle u, v\rangle  =\tr_{S/X}\left(\frac{\hat\sigma(u)v}{d\pi}\right) \ .
\end{equation}
Since $\hat\sigma$ squares to $-1$, we have
\begin{equation*} \label{eqn:minus}
\left(\frac{\hat\sigma(u)v}{d\pi}\right) (z) = - \left(\frac{\hat\sigma(v)u}{d\pi}\right) (\sigma(z)) \ ,
\end{equation*}
and hence since $\tr_{S/X}(\, \cdot\, )=\tr_{\overline S/X}\tr_{S/\overline S}(\, \cdot\, )$, the pairing
 $\langle\ ,\ \rangle$ is skew. The fact that it is nondegenerate follows as in \cite{Hitchin:87b}.

Let $\mathcal{M}^0_{\PSp(2m)}$ be the moduli space of semistable Higgs bundles for $\PSp(2m)$ that lift to  $\Sp(2m)$ bundles, and let $h_{\PSp(2m)}: \mathcal{M}^0_{\PSp(2m)}\rightarrow \Bcal(N)$ be the Hitchin map. Since the Higgs field take values in the adjoint bundles, the natural projection of $\Sp(2m)\rightarrow \PSp(2m)$ give a natural projection $\mathcal{M}_{\Sp(2m)}\rightarrow \mathcal{M}^0_{\PSp(2m)}$. Now Theorem \ref{theorem:spspectral} has  the following consequence:
\begin{corollary}\label{cor:spectralpsp}
	The fibers of the Hitchin map $h_{\PSp(2m)}$ are in one to one correspondence with points $P(S,\overline{S})/J_2(X)$, where $J_2(X)$ acts through pulling back by $\pi: S\to X$. 
\end{corollary}

\subsection{Spectral data for odd orthogonal bundles}\label{sec:spectraloddorthogonal}Let $V$ be a vector bundle of rank $2m+1$ with a nondenegerate symmetric bilinear $(\, ,\, )$ form along with a trivialization $\det {V}\simeq \mathcal{O}_X$. Let $\Phi: V\rightarrow V\otimes K_X$ be a Higgs field satisfying
$( v,\Phi w)+(\Phi v,w)=0$. Then 
$$\det(\lambda-\Phi)=\lambda(\lambda^{2m}+b_{2}\lambda^{2m-2}+
\dots+b_{2m})\ .$$
As in the case of the symplectic bundle, we assume that the zeros $Z(b_{2m})$ of $b_{2m}$ are simple and  that $b_{2m-2}$ is nonvanishing on  $Z(b_{2m})$.  In \cite[Sec.\ 4.1]{Hitchin:07} Hitchin shows that the one dimensional zero eigenspace of $\Phi$ generates the line bundle $V_0$  isomorphic to $K_X^{-m}$,  and that the quotient $V_1=V/V_0$ is of the form $E\otimes K_X^{1/2}$, where $E$ is a symplectic bundle. The symplectic form $\langle\, ,\, \rangle$  on $E$ is induced by the formula
$\langle v_1,v_2\rangle=(v_1,\Phi v_2)$, and the Higgs field $\Phi$ restricted to $E\otimes K_X^{1/2}$ induces a symplectic Higgs field. We  again define the spectral curve  $S$ 
by \eqref{eqn:spectral-curve}, and so the orthogonal bundle gives
rise to $L\in P(S,\overline S)$.

Going in the other direction, starting with a symplectic bundle $(E,\Phi')$ we define an orthogonal structure on $V=V_0\oplus V_1$ by using $b_{2m}$ on $V_0$ and 
\begin{equation}\label{eqn:odd-orthogonal-pairing}
(u, v)  =\tr_{S/X}\left(\frac{\hat\sigma(u)v\cdot \lambda^{-1}}{d\pi}\right) 
\end{equation}
(see \eqref{eqn:symp-pairing}) on $V_1$. Because the section $\lambda$ is odd, this becomes and even pairing.
This only defines the orthogonal structure on $V$ away from the ramification locus, however, and extending it  to $X$
requires more information. Since the details are not important for this paper, we simply state the result.


\begin{theorem}[{\cite{Hitchin:07}}] \label{thm:odd-orthogonal}
	Let $(E,\Phi')$ be a generic symplectic Higgs bundle of rank $2m$. Then an associated $\SO(2m+1)$ Higgs bundle is determined by a vector $e_a \in E_a\otimes K_a^{m-1/2}$ for each point $a \in Z(b_{2m})$ satisfying a certain compatibility condition with $\Phi'$.
   Moreover, the generic fibers of $h_{\SO(2m+1)}^{\pm}$ for each connected component $\mathcal{M}^{\pm}_{\SO(2m+1)}$ are isomorphic to $P(S,\overline{S})/ J_2(\overline{S})$, where the $J_2(\overline S)$-action is via the pullback by $p: S\to \overline S$.
\end{theorem}

\subsection{Spectral data for even orthogonal bundles} \label{sec:spectralevenorthogonal}
In this case, a point in $\Bcal(N)$ is of the form: $\vec b=(b_2, \ldots, b_{2m-2}, p_m)$, $N=2m$.
We assume  $p_m$ has simple zeros at $Z(p_m)$ and that $b_{2m-2}$ is nonvanishing on $Z(p_m)$. The curve defined analogously to \eqref{eqn:spectral-curve} is:
$$
S' :=\left\{ w\in {\rm tot} (K_X) \mid \lambda^{2m}+\pi^\ast(b_2)\lambda^{2m-2}+\cdots +\pi^\ast(b_{2m-2})\lambda^2+\pi^\ast(p_m^2)=0 \text{ at } w \right\}\ .
$$
The zeros $Z(p_m)$ (which we view also as points in $S'$) are singularities of $S'$.
 With the assumption above these are the only singularities, and they consist of $2m(g-1)$ ordinary double points. Let $\sigma$ be the involution on $S'$ sending $\lambda$ to $-\lambda$. The fixed points of $\sigma$ are exactly the  singularities. If $S$ denotes the normalization of $S'$, $\mu: S\to S'$, then
 since the singularities  are double points, $\sigma$ extends to an involution of $S$. 
The relevant diagram now is the following:
\begin{equation}
\begin{split}
\xymatrix{
&S\ar[dl]_\mu \ar[d]^{p} \ar@/^2pc/[dd]^{\pi}    \\
	S' \ar[r]^{\ p'} \ar[dr]_{\pi'} & \overline{S} \ar[d]^{\overline\pi}\\
	 &X
}
\end{split}
\end{equation}
The double covering $p: S\to \overline S$ is unramified, and hence it is determined by a line bundle $\Lcal\in J_2(\overline S)$.  We will need the following:
\begin{lemma} \label{lem:L}
The line bundle $\Lcal\to\overline S$ is in the kernel of the norm map $\nm_{\overline S/X}:J(\overline S)\to J(X)$.
\end{lemma}

\begin{proof}
We prove this by computing $\det\pi_\ast\Ocal_S$ in two different ways. First, the normalization gives an exact sequence: 
$$
0\lra \Ocal_{S'}\lra \mu_\ast \Ocal_S\lra \Ocal_{Z(p_m)}\lra 0\ .
$$
Since $\pi=\pi'\circ\mu$, this implies
\begin{equation} \label{eqn:seq-first}
0\lra (\pi')_\ast \Ocal_{S'}\lra \pi_\ast \Ocal_S\lra \Ocal_{Z(p_m)}\lra 0\ .
\end{equation}
Now from general facts about spectral curves, we get $\det (\pi')_{\ast}\Ocal_{S'}\simeq K_X^{-2m(m-1)}$ (cf.\ \cite[Sec.\ 3]{BNR}).
So from \eqref{eqn:seq-first} we have
\begin{equation} \label{eqn:det-first}
\det\pi_\ast\Ocal_S\simeq K_X^{-m(2m-1)}\otimes K_X^m=K_X^{-2m(m-1)}\ .
\end{equation}
On the other hand, by the definition of $\Lcal$, $p_\ast\Ocal_S=\Ocal_{\overline S}\oplus \Lcal$. It follows that
\begin{align}
\pi_\ast\Ocal_S&=\overline\pi_\ast\left(p_\ast\Ocal_S\right)=\overline \pi_\ast\Ocal_{\overline S}\oplus\overline\pi_\ast\Lcal\ , \notag \\
\det\pi_\ast\Ocal_S&\simeq \left(\det\overline \pi_\ast\Ocal_{\overline S}\right)^{2}\otimes\nm_{\overline S/X}\Lcal\ .
\label{eqn:det-second}
\end{align}
As before (except now note that $\overline S\subset {\rm tot} K_X^2$), we have
$\overline\pi_\ast\Ocal_{\overline S}\simeq \oplus_{i=0}^{m-1} K_X^{-2i}$. Plugging this into \eqref{eqn:det-second}
 we obtain
 $$
\det\pi_\ast\Ocal_S\simeq K_X^{-2m(m-1)}\otimes \nm_{\overline S/X}\Lcal\ .
 $$
 Now comparing this with \eqref{eqn:det-first}, we conclude that $\nm_{\overline S/X}\Lcal\simeq\Ocal_X$.
\end{proof}

Returning to the spectral data, as in the symplectic and odd orthogonal cases, for $L\in K(S,\overline S)$ we let $V=\pi_\ast U$, where $U=L\otimes (K_S\otimes \pi^\ast K_X^{-1})^{-1/2}$.
The pairing is defined by
\begin{equation}
\label{eqn:even-orthogonal-pairing}
(u, v)  =\tr_{S/X}\left(\frac{\hat\sigma(u)v}{d\pi}\right) 
\end{equation}
 as in \eqref{eqn:symp-pairing}, except that now $\widehat \sigma$ squares to the identity so that the pairing is symmetric.
Let $P(S,\overline{S})\subset K(S,\overline S)$ denote the connected component of the identity of $\ker\nm_{S/\overline S}$. Then we have the following theorem due to N. Hitchin \cite{Hitchin:87b} 
\begin{theorem}
	\label{thm:even-orthogonal}
  The correspondence described above identifies a generic  fiber of the Hitchin map $h_{\SO(2m)}^{\pm}$  with  $P(S,\overline S)$
 .
\end{theorem}

The moduli space of $\PSO(2m)$-Higgs bundles has four connected components. Let $\mathcal{M}_{\PSO(2m)}^0$ denote the neutral component 
consisting of those bundles which  lift to $\Spin(2m)$-bundles. 
As an easy corollary of Theorem \ref{thm:even-orthogonal} we get a description of the spectral data:
\begin{corollary}
	\label{cor:spectralpso}The generic fibers of the Hitchin map $h^0_{\PSO(2m)}:\mathcal{M}_{\PSO(2m)}^0\rightarrow \mathcal{A}(m)$ are in one-to-one correspondence with elements of the abelian variety $P(S,\overline{S})/J_2(X)$.
\end{corollary}

\section{Spin structures from spectral data}

\subsection{Case of special spectral data}

\begin{lemma}
Fix generic $\vec b\in \Bcal(N)$.
Consider the orthogonal bundles $V_c$:
\begin{itemize}
\item for $N=2m+1$:
$$
V_c=K_X^{-m}\oplus K_X^{-m+1}\oplus \cdots\oplus \Ocal_X\oplus\cdots\oplus K_X^{m-1}\oplus K^m\ ,
$$
with orthogonal structure given by the pairing of $K^{-j}$ with $K^j$, and $\Ocal_X$ an orthogonal subbundle;
\item  for $N=2m$: 
$$
V_c=K_X^{-m}\oplus K_X^{-m+1}\oplus \cdots\oplus \Ocal_X\oplus\cdots\oplus K_X^{m-1}\oplus K^m\oplus \Ocal_X\ ,
$$
where the last factor  $\Ocal_X$ is also an orthogonal subbundle. 
\end{itemize}
Then $V_c$ admits a Higgs field $\Phi_c$ such that $(V_c,\Phi_c)$ is a stable $\SO(N)$-Higgs bundle in the fiber over $\vec b$.
\end{lemma}

\begin{proof}
This follows from Hitchin's construction via the principal $\slfrak(2)$-embedding for split real forms (see \cite{Hitchin:92} and also \cite[Sections 8.4-5]{Aparicio}). 
\end{proof}

By  Theorems \ref{thm:odd-orthogonal} and \ref{thm:even-orthogonal}, there is a line bundle $L_c\in P(S,\overline S)$  such that $L_c$ gives spectral data for $(V_c,\Phi_c)$.

\begin{lemma} \label{lem:canonical-spin}
Recall that we have fixed a theta characteristic $K_X^{1/2}$. Then $K_X^{1/2}$ also  determines  a lift of $(V_c,\Phi_c)$ to $\Spin(N)$-Higgs bundles. 
\end{lemma}

\begin{proof}
The bundles $V_c$ admit quasi-isotropic decompositions $W^{+}\oplus W^{-} \oplus\Ocal_X$ (odd case) and $W^{+}\oplus W^{-}$ (even case), where 
$\det W^+=K_X^{m(m+1)/2}$. 
Indeed, in the $N$ odd case we take $W^+=\oplus_{i=1}^m K^i$. For $N$ even, we add to this a choice of isotropic line in $\Ocal_X\oplus \Ocal_X$. It is well-known that a choice of square root of $\det W^+$ determines a lift to $\Spin$, and such a root is determined by $K_X^{1/2}$ (cf.\ \cite{Hitchin:16}).
\end{proof}

\subsection{Application of the Hecke transformation}\label{sec:hecke}
In Section \ref{sec:hecke-orthogonal} we described how an orthogonal bundle $V\to X$ with a choice of isotropic line $\tau$ at a point $p\in X$ gives rise to a new orthogonal bundle $V^\iota$.   In this section, we relate the spectral data of these orthogonal bundles under this transformation. We have the following.

\begin{proposition} \label{prop:hecke}
For $L\in K(S,\overline S)$, let $V_{L}=\pi_\ast U$,  $U=L\otimes (K_S\otimes \pi^\ast K_X^{-1})^{-1/2}$, be the orthogonal bundle associated to $L$ by Theorems \ref{thm:odd-orthogonal} and \ref{thm:even-orthogonal} for $N$ odd and even, respectively. Choose a point $p\in S$ outside the ramification locus,
 and let $\widetilde L=L\otimes \Ocal_S(p)\otimes \Ical_{\sigma(p)}$. The fiber of $U$ at $p$ corresponds to an isotropic line $\tau$ in the fiber of $V_{L}$ at $\pi(p)$. 
Then $V_{\widetilde L}$ is isomorphic to the orthogonal bundle $V_{L}^\iota$ in Definition \ref{def:iota}.
\end{proposition}

\begin{proof}
  Let $\tau_1=V_{L}/\tau^\perp$. It is a skyscraper sheaf supported at $\pi(p)$ of length $1$. Denote the orthogonal structure on $V_L$ by $(\, ,\, )$. 
By definition of the pairing \eqref{eqn:odd-orthogonal-pairing} and \eqref{eqn:even-orthogonal-pairing}, 
$\tau_1$ can be identified with the fiber of $U$ at $\sigma(p)$. Under the direct image,
the sheaf map $V_{L}\to \tau_1$ given by $e\mapsto (e, \tau)$ corresponds
  to evaluation $U\to U_{\sigma(p)}$. In other words, the direct image of the exact sequence
$$
0\lra U\otimes \Ical_{\sigma(p)}\lra U\lra U_{\sigma(p)}\lra 0
$$
is 
$$
0\lra V_{L}^\flat \lra V_{L}\lra \tau_1\lra 0\ .
$$
With respect to the orthogonal structure,  $V_{L}^\sharp=( V_{L}^\flat )^\ast $ is the direct image of $U\otimes\Ocal_S(p)$. Now 
 the direct image of the exact sequence
$$
0\lra U\lra U\otimes\Ocal_S(p)\lra U_{p}\lra 0
$$
is 
$$
0\lra V_{L} \lra V_{L}^\sharp\lra \tau_2\lra 0\ .
$$
By definition, $V_{L}^\iota$ is the kernel of the induced map $V_L^\sharp\to V_{L}/V_{L}^\flat\simeq \tau_1$, and it follows that
$V_{L}^\iota=\pi_\ast\left(U\otimes\Ocal_S(p)\otimes \Ical_{\sigma(p)}\right)$.
\end{proof}

\begin{corollary} \label{cor:hecke}
Let $L\in K(S,\overline S)$, $L=L_c\otimes M\otimes \sigma^\ast(M^\ast)$ for $M\in J(S)$. Then for any generic choice of divisor $\Div M$, $V_L$ is isomorphic to the Hecke transform $V_c^\iota$ at $\pi(\Div M)$. 
\end{corollary}
\begin{proof}
Write the divisor $D$ of $M$ as:
$$
D=p_1+\cdots + p_r -q_1-\cdots -q_s\ .
$$
Then 
$$
-\sigma(D)=-\sigma(p_1)-\cdots - 
\sigma(p_r) +\sigma(q_1)+\cdots +\sigma(q_s)
$$
is a divisor of $\sigma^\ast(M^\ast)$, and so $M\otimes \sigma^\ast(M^\ast)$ has divisor
$$
\bigotimes_{i=1}^r\left(\Ocal_S(p_i)\otimes \Ical_{\sigma(p_i)}\right)\otimes 
\bigotimes_{i=1}^s\left(\Ocal_S(\sigma(q_i))\otimes \Ical_{q_i}\right)\ .
$$
Now apply Proposition \ref{prop:hecke} repeatedly.
\end{proof}

Using the results of Section \ref{sec:hecke-spin} (see Proposition \ref{prop:hecke} and recall the convention in Remark \ref{rmk:sign}), we also have

\begin{corollary} \label{cor:spin-hecke}
 In addition to the hypothesis of Proposition \ref{prop:hecke}, suppose that $V_L$ has a lift to an $\SC$-bundle $P_L$. Then
this determines a lift $P_{\widetilde L}$ of $V_{\widetilde L}$ to an $\SC$-bundle with
$$
\Nm(P_{\widetilde L})=\Nm(P_{L})\otimes \Ocal_X(p)\ .
$$
In particular, if $L=L_c\otimes M\otimes \sigma^\ast(M^\ast)$, $M\in J(S)$,
 a choice of spin structure on $V_{L_c}$ determines a lift of $V_L$ to an $\SC$-bundle $P_M$ with 
$\Nm(P_M)=\nm_{ S/X}(M)$. 
\end{corollary}

\begin{remark}
Implicit in Corollary \ref{cor:hecke} is the following: if we modify the choice of divisor $D$ of $M$ by a generic meromorphic function $f$ then there is a natural 
isomorphism of the orthogonal bundles obtained by Hecke transformations on $\pi(D)$ and $\pi(D+\Div(f))$.
Indeed, multiplication of sections of $U$ (in the proof of Proposition \ref{prop:hecke}) by $f/\sigma^\ast(f)$, is an isometry with respect to the pairing \eqref{eqn:odd-orthogonal-pairing} or \eqref{eqn:even-orthogonal-pairing}. Furthermore, 
in an appropriate local frame this isometry has the form of $\zeta$ in \eqref{eqn:zeta}, and so the
 $\SC$-bundles obtained in Corollary \ref{cor:spin-hecke} are similarly isomorphic.
\end{remark}

\subsection{Proof of Theorem \ref{thm:main}} \label{sec:proof}
Let $P$ be a spin Higgs bundle in the fiber over $\vec b$. According to Theorems \ref{thm:odd-orthogonal} and \ref{thm:even-orthogonal}, 
the underlying orthogonal bundle to $P$ is of the form
$V_L$ for spectral data $L\in P(S,\overline S)$. 
Recall from Lemma \ref{lem:canonical-spin} that a choice of theta characteristic determines a  spin structure on $V_c$.
By Corollary \ref{cor:spin-hecke}, if we write $L=L_c\otimes M^2$
for $M\in K(S,\overline S)$, then there is a lift of $V_L$ to a spin bundle $P_M$. We must check the dependence of this lift on the choice of $M$.
By Theorems \ref{thm:odd-orthogonal} and \ref{thm:even-orthogonal}, the ambiguity in the choice of $M$ is the action of $J_2(\overline S)$.  So
consider $M\otimes p^\ast N$, for $N\in J_2(\overline S)$. Let $N(t)$ be a family in $J(\overline S)$, $N(0)=\Ocal_{\overline S}$ and $N(1)=N$, and set $M(t)=M\otimes p^\ast N(t)\in J(S)$.
Thus $L=L_c\otimes M(t)\otimes \sigma^\ast(M^\ast(t))$ for all $t$.
By Corollary \ref{cor:spin-hecke}, we obtain a family of lifts of $V_L$ to $\SC$-bundles $P_{M(t)}$ with 
\begin{equation} \label{eqn:Q}
\Nm\left(P_{M(t)}\right)=\nm_{S/X}(M)\otimes \nm_{S/X}(p^\ast N(t))
=\left(\nm_{\overline S/X} N(t)\right)^2\ .
\end{equation}
In \eqref{eqn:Q} we have used two facts: first, since $M\in \ker \nm_{S/\overline S}$, and $\nm_{S/X}=\nm_{\overline S/X}\circ\nm_{S/\overline S}$, we have $\nm_{S/X}(M)=\Ocal_X$; and second, 
$$
\nm_{S/X}\left(p^\ast N(t)\right)=
\nm_{\overline S/X}\left(\nm_{S/\overline S}p^\ast N(t)\right)=
\nm_{\overline S/X}\left(N^2(t)\right)=
\left(\nm_{\overline S/X}(N(t))\right)^2 \ .
$$
Now it follows from \eqref{eqn:Q} that
$$
\Nm\left(P_{M(t)}\otimes\left(\nm_{\overline S/X} N(t)\right)^\ast \right)=
\Ocal_X\ .
$$
But then $P_{M(t)}\otimes\left(\nm_{\overline S/X} N(t)\right)^\ast$
is a family of  spin bundles that lift the fixed orthogonal bundle $V_L$. The set of such lifts is finite, so the family  is necessarily constant. Evaluating at $t=0$ and $t=1$, we find:
$$
P_{M\otimes p^\ast N}\simeq P_M\otimes \nm_{\overline S/X} N\ .
$$
Given one lift of $V_L$ to a spin bundle, the others are obtained by tensoring by elements of $J_2(X)$. From the above, a change of $M$ is equivalent  to the action of $J_2(\overline S)$ on $J_2(X)$.
Since $\nm_{\overline S/X}: J_2(\overline S)\to J_2(X)$ is surjective (cf.\ the next section), $P=P_{M\otimes p^\ast N}$ for some choice of $N$.
The proof of the theorem for spin bundles thus follows. The proof in the twisted case follows similarly by applying the first part of Corollary \ref{cor:spin-hecke} and using the same argument as above.

\section{Duality} \label{sec:duality}

We continue with the same notation as in the previous sections.
We may regard $J_2(X)$ as a subgroup of $J_2(\overline S)$ by pullback $\pi^\ast$, and similarly $J_2(\overline S)$ as a subgroup of $J_2(X)$ by the norm map $\nm_{\overline S/X}$. These are dual operations. Using this fact,  and dualizing  the exact sequence
\begin{equation}
\label{eqn:torisionnormmap}
1\lra J_2(X)\stackrel{\pi^\ast}{\lra} J_2(\overline S)\lra J_2(\overline S)/J_2(X)\lra 1\ ,
\end{equation}
we see that 
$
\left[ J_2(\overline S)/J_2(X)\right]^\vee\subset J_2(\overline S)^{\vee}.
$
Since the pullback $\pi^*$ is injective it follows that $\nm_{\overline S/X}: J_2(\overline{S})\rightarrow J_2({X})$ is surjective. Recall that $K(S,\overline{S})=\ker\operatorname{Nm}_{S/\overline{S}}$, whereas the Prym $P(S,\overline S)$ is the connected component containing $\Ocal_S$. 

\begin{remark} \label{rem:fact}
In the case where $S\rightarrow \overline{S}$ is \'etale, let $H_0$ denote the kernel of the map $J(\overline S)\to J(S)$, and recall that $\Lcal\in J_2(\overline S)$ defines the cover. Then $H_0=\{\Ocal_{\overline S}, \Lcal\}$. Moreover, if $H_1\subset J_2(\overline S)$ denotes the annihilator of $\Lcal$ with respect to the Weil pairing, then pulling back from $\overline S$ to $S$ gives an identification of the two torsion points of the Prym  $P_2(S,\overline S)\simeq H_1/H_0$  (cf. \cite{Mumford:74}).
\end{remark}

Now consider the following variety:
\begin{align*}
K(S ,\overline{S})\times_{J_2(\overline{S})}&J_2(X):=\\
&\{(a,b)\in K(S ,\overline{S})\times J_2(X)\}/\{(a\cdot s,b\cdot \nm_{\overline S/X}(s))\sim (a,b)\mid s\in J_2(\overline{S})\}
\end{align*} Observe that $K(S,\overline{S})\times_{J_2(\overline{S})}J_2(X)$ can be realized as a group quotient of $K(S,\overline{S})\times J_2(X)$ and hence has a natural group law. 
Now we define a map $\iota$ by:
$$\iota:{P(S ,\overline S)}\lra K(S ,\overline{S})\times_{J_2(\overline{S})}J_2(X): 
a \mapsto [(a,1)]\ .
$$  



\begin{lemma}\label{lem:connected}
The map $\iota$ induces the following: 
\begin{enumerate}
	\item If $S\rightarrow \overline{S}$ is ramified, then $$
	\frac{P(S ,\overline S)}{\left[ J_2(\overline S)/J_2(X)\right]^\vee}\simeq K(S ,\overline{S})\times_{J_2(\overline{S})}J_2(X)\ .
	$$
	\item If $S\rightarrow \overline{S}$ is \'etale, then 
	$$\frac{P(S,\overline{S})}{\left[ (H_1/H_0)/J_2(X)\right]^\vee}\simeq  K(S ,\overline{S})\times_{J_2(\overline{S})}J_2(X)\ .$$
\end{enumerate}
In particular, in both cases, $K(S ,\overline{S})\times_{J_2(\overline{S})}J_2(X)$ is an abelian variety. 
\end{lemma}

\begin{proof}
First, suppose $S \rightarrow \overline{S}$ is ramified. 
Let $[(a,b)]\in K(S ,\overline{S})\times_{J_2(\overline{S})}J_2(X)$. Since   $\operatorname{Nm}_{\overline{S}/X}$ is surjective we can rewrite any representative in the form $(a_1,1)$, where $a_1=as$ and $s$ is an element of $J_2(S)$ such that $\nm_{S/X}(s)=b$. In particular,   $\iota$ is surjective.
 On the other hand,  $\ker\iota$ is clearly given by the kernel of  $\nm_{\overline S/X}: J_2(\overline{S})^{\vee}\rightarrow J_2(X)^{\vee}$. This is precisely ${\left[ J_2(\overline S)/J_2(X)\right]^\vee}$ by the exact seuquence \eqref{eqn:torisionnormmap}.
 
 Now we consider the case where $S\rightarrow \overline{S}$ is \'etale. By \cite[Prop.\ 11.4.3]{BRHL}, the pullback $J(X)\rightarrow J(S)$ is injective. This implies that the line bundle $\mathcal{L}$ defining the \'etale cover is not in the image of $J_2(X)$ under pullback. 
We claim that the image of $J_2(X)$ lies in $H_1$. Indeed, with respect to the Weil pairing $\langle\, ,\, \rangle$, if $M\in J_2(X)$ then
$$
\langle \overline \pi^\ast M, \Lcal\rangle_{\overline S}
=\langle M, \nm_{\overline S/X}\Lcal\rangle_X=1\ ,
$$
since $\operatorname{Nm}_{\overline{S}/X}(\mathcal{L})=\mathcal{O}_{X}$
by Lemma \ref{lem:L}; hence, the claim.
It follows that $J_2(X)$ injects into $H_1/H_0$ and also that $H_1/H_0$ surjects to $J_2(X)$ under the map $\operatorname{Nm}_{\overline{S}/X}$. If $[(a,b)]\in K(S ,\overline{S})\times_{J_2(\overline{S})}J_2(X)$ is such that $a \in P(S,\overline{S})$, then we can find an element $s\in P_2(S,\overline S)$ such that $\operatorname{Nm}_{\overline{S}/X}(s)=b$. Then $[a,b]=[as,1]$, and clearly $as \in P(S,\overline{S})$. If, on the other hand,  $a\in K(S,\overline{S})\backslash P(S,\overline{S})$, then there is  $\zeta$ of $J_2(\overline{S})\backslash H_1$ such that $a\zeta\in P(S,\overline{S})$. 
By modifying $\zeta$ with elements in  $H_1$ and using the surjectivity of $\nm_{\overline S/X}:H_1/H_0\to J_2(X)$, we can furthermore arrange that $\nm_{\overline S/X}(\zeta)=\Ocal_X$.
Then we are done with the proof of surjectivity of $\iota$, since in this case $[a,b]=[a\zeta,b]$. Finally, $a\in \ker\iota$  implies that $[a,1]=[1,1]$ in $K(S,\overline{S})\times_{J_2(\overline{S})}J_2(X)$; hence, $a$ is two torsion and $\operatorname{Nm}_{\overline{S}/X}(a)=1$, and so by Remark \ref{rem:fact} the  kernel of $\iota$ is identified with  $((H_1/H_0)/J_2(X))^{\vee}$.
\end{proof}

\begin{lemma} \label{lem:duality}
We have the following isomorphism of abelian varieties.
$$
\left[P(S ,\overline S)/J_2(X)\right]^\vee\simeq K(S ,\overline{S})\times_{J_2(\overline{S})}J_2(X)\ .
$$
\end{lemma}
\begin{proof}
	Let $f: A\rightarrow B$ be an isogeny of abelian varieties.
	 $A^{\vee}$ and $B^{\vee}$ be the corresponding dual abelian varities. 
	 Then there exists an isogeny of dual abelian varieties with the following exact sequence: 
	 \begin{equation}
	 \label{eqnabel}
	  1\rightarrow (\ker f)^{\vee}\rightarrow B^{\vee}  \rightarrow A^{\vee}\rightarrow 1\ ,
	 \end{equation}
	 where $\ker f$ and $(\ker f)^{\vee}$ are Cartier dual of each other. Consider the case, when $S\rightarrow \overline{S}$ is not \'etale.	Applying the above with $A=P(S,\overline{S})/J_2(X)$ and $B=P(S,\overline{S})/J_2(\overline{S})$, we obtain an exact sequence
	$$1\rightarrow \big( J_2(\overline{S})/J_2(X)\big)^{\vee}\rightarrow \big(P(S ,\overline{S})/J_2(\overline{S})\big)^{\vee}\rightarrow \big(P(S ,\overline{S})/J_2(X)\big)^{\vee}\rightarrow 1\ .$$
		It is well-known   that $\big(P(S ,\overline{S})/J_2(\overline{S})\big)^{\vee}\simeq P(S ,\overline{S})$. Hence, the result follows from  Lemma \ref{lem:connected}.
		
		If $S\rightarrow \overline{S}$ is \'etale, then we put $B=H_1/H_0$ and we get an exact sequence 	$$1\rightarrow \big( (H_1/H_0)/J_2(X)\big)^{\vee}\rightarrow \big(P(S ,\overline{S})/(H_1/H_0)\big)^{\vee}\rightarrow \big(P(S ,\overline{S})/J_2(X)\big)^{\vee}\rightarrow 1\ .$$
     In this case,  $P(S,\overline{S})$ is principally polarized, and hence 
		$$\big(P(S ,\overline{S})/(H_1/H_0)\big)\simeq P(S ,\overline{S})\simeq P(S,\overline{S})^{\vee}\ .$$ Now using  Lemma \ref{lem:connected} completes the proof.
	\end{proof}

\begin{proof}[Proof of Corollary \ref{cor:connected}] Immediate from Theorem \ref{thm:main} and Lemma \ref{lem:connected}.
\end{proof}

\begin{proof}[Proof of Theorem  \ref{thm:duality}] Immediate from Theorem \ref{thm:main}, Lemma \ref{lem:duality}, and Corollaries \ref{cor:spectralpsp} and \ref{cor:spectralpso}.
\end{proof}

\bibliographystyle{amsplain}
\bibliography{../papers}

\end{document}